\documentclass[journal]{IEEEtran}
\usepackage{amssymb}
\usepackage{amsthm}
\usepackage[cmex10]{amsmath}
\usepackage{graphicx}
\usepackage{colortbl}
\usepackage{cite}
\usepackage{stfloats}

\newtheorem{theorem}{Theorem}

\newtheorem{remark}{Remark}
\newtheorem{lemma}{Lemma}
\newtheorem{assumption}{Assumption}

\newtheorem{problem}{Problem}

\ifCLASSINFOpdf
\else
\fi
\begin{document}
%
\title{Optimal Decentralized Control with Asymmetric Partial Information Sharing
}
%
%
%

\author{Xiao~Liang,~Qingyuan Qi,~Huanshui~Zhang,~\emph{Senior~Member,~IEEE,}~and~~Lihua Xie~\emph{Fellow,~IEEE,}
\thanks{This work is supported by the Taishan Scholar Construction Engineering by Shandong Government, the National Natural Science Foundation of China (61903233, 61120106011, 61903233, 61633014, 61403235, 61573221).}
\thanks{H. Zhang (hszhang@sdu.edu.cn) is with School of Control Science and Engineering, Shandong University, Jinan, P.R.China. L. Xie (ELHXIE@ntu.edu.sg) is with the School of Electrical and Electronic Engineering, Nanyang Technological University, Singapore 639798. Q. Qi is with Institute of Complexity Science, College of Automation, Qingdao University, Qingdao 266071, China. X. Liang (liangxiao\_sdu@163.com) is with College of Electrical Engineering and Automation, Shandong University of Science and Technology, Qingdao, Shandong, P.R.China 250061.}
}

%
%

\markboth{}%
{Shell \MakeLowercase{\textit{et al.}}: Bare Demo of IEEEtran.cls for IEEE Journals}
%



\maketitle

\begin{abstract}
This paper considers the optimal decentralized control for networked control systems (NCSs) with asymmetric partial information sharing between two controllers. In this NCSs model, the controller 2 (C2) shares its observations and part of its historical control inputs with the controller 1 (C1), whereas C2 cannot obtain the information of C1 due to network constraints. We present the optimal estimators for C1 and C2 respectively based on asymmetric observations. Since the information for C1 and C2 are asymmetric, the estimation error covariance (EEC) is coupled with the controller which means that the classical separation principle fails. By applying the Pontryagin's maximum principle, we obtain a solution to the forward and backward stochastic difference equations. Based on this solution, we derive the optimal controllers to minimize a quadratic cost function. Combining the optimal controllers with the EEC, the controller C1 is decoupled from the ECC. It should be emphasized that the control gain is dependent on the estimation gain. What's more, the estimation gain satisfies the forward Riccati equation and the control gain satisfies the backward Riccati equation which makes the problem more challenging. We propose iterative solutions to the Riccati equations and give a suboptimal solution to the optimal decentralized control problem.
\end{abstract}


%
\IEEEpeerreviewmaketitle

\section{Introduction}
Networked control systems (NCSs), where control loops are closed via communication networks such that control and measurement signals can be exchanged between system constituents (sensors, estimators, controllers and actuators), have received increasing interests \cite{R1,R2}. Compared with classical point-to-point feedback control systems, NCSs have huge advantages, such as lower cost, easy maintenance and higher flexibility \cite{R3,R4}. Centralized configuration and decentralized configuration are two important configurations of NCSs.

Centralized configuration, in spite of the existence of multiple sensors and actuators, in some cases, can be regarded as a single feedback loop configuration where all measurements are delivered to the controller. The control decision is then sent to the dedicated actuator. Since the controller can obtain all measurements, the derived controller may be globally optimal. Many results have been given based on this configuration \cite{R5,R6,R7}. By using the optimal encoder-decoder design, \cite{R5} presents the optimal controller design for an arbitrary measurement packet-dropout pattern. In virtue of the method of completing square, \cite{R7} solves the linear quadratic regulation control for NCSs with input delay and measurement packet-dropout. Nevertheless, this configuration bears a few disadvantages: (i) the system is prone to being shut down completely due to the failure of the central processing unit (CPU); (ii) there is a high cost for CPU to gather information from all sensors.

In view of the limitations of centralized configuration, the other configuration of NCSs, decentralized configuration has gain continuous attention in recent years \cite{R8,R9} and references therein. For a large scale NCS, the control decision may not be made by a single controller, but multiple controllers that access different information about the system make decisions together. This configuration reduces the point of failure risk caused by centralized control, alleviates the computation burden of controllers, and decreases the implementation complexity of NCSs. Generally, in decentralized configuration, non-linear control decisions may provide better performance than linear control decisions \cite{R10,R11}. In other words, linear control decisions are not globally optimal. It is noted that under some special information structures, such as partially nested information structure, pioneered by \cite{R12}, linear control decisions may be globally optimal \cite{R13}. Recently, \cite{R14} proposes a novel general structure of partial historical sharing, in which controllers share part of their historical information (historical observations and controls). This structure comprises a huge class of decentralized configuration where controllers can exchange information. Based on this structure, under the assumption that controllers/system satisfy some special forms, \cite{R15}, \cite{R23} gives the optimal linear control strategies by using the common information approach. Inspired by this work, \cite{R16,R17} investigate the control of NCSs with local and remote controllers. In \cite{R16,R17}, the optimal controllers are derived by using maximum principle and dynamic programm respectively. However, both the papers \cite{R16,R17} assume that the local controller can access the exact state information of the system which is generally not real in practice. The measured state is inevitable to be corrupted by noises.

Inspired by \cite{R16,R17}, this paper studies a general decentralized configuration of NCSs as shown in Fig. 1. The state is observed by sensor 1 and sensor 2 as standard observations $y_k^1$ and $y_k^2$ respectively. The controller 2 (C2) shares its observations and historical control inputs $\{y_0^2,\ldots,y_k^2,u_0^2,\ldots,u_{k-1}^2\}$ with controller 1 (C1), but C1 does not share its information with C2 due to network constraints. This results in information asymmetry between C1 and C2.  Estimator 1 estimates the state by using the observations of itself and C2, and delivers its estimate to C1. The two controllers control the plant simultaneously.
\begin{figure}[htbp]
  \begin{center}
  \includegraphics[width=0.41\textwidth]{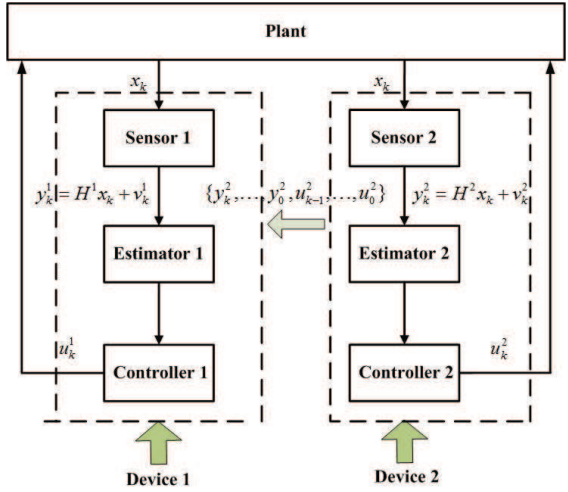}
  \caption{Overview of NCSs.} \label{fig:digit}
  \end{center}
\end{figure}

In this paper, we consider the optimal decentralized control for NCSs with  asymmetric information and partial historical sharing. Firstly based on the asymmetric observations, optimal estimators are derived for C1 and C2 respectively. Due to the adaptability of the controllers, the estimation error covariance (EEC) of estimator 2 is coupled with C1. By applying the Pontryagin's maximum principle, the solution to the forward and backward stochastic difference equations (FBSDEs) is presented. By making use of this solution, the optimal controllers are given in terms of two coupled Riccati equations. Combining the optimal controllers with the EEC, the controllers are decoupled from the EEC. However, the control gain is coupled with the estimator gain. Furthermore, the control gain satisfies a backward Riccati equation while the estimator gain is related to a forward Riccati equation. To address this problem, we introduce a novel iterative method  to solve the coupled backward and forward Riccati equations in the infinite horizon. Finally, numerical examples are given to illustrate that estimators are stable, the regulated states are bounded in the mean square sense, and the algorithm is valid.

The contributions of this paper are as follows:

(i) So far as we know, it is the first time to give the explicit expressions of the optimal decentralized controllers for NCSs with asymmetric information and partial history sharing without the restriction on the form of controllers.

(ii) As is well-known that the classical separation principle fails when the ECC is dependent on the control input \cite{R18}. In this paper, we break through this obstacle and succeed to separate the control input from the ECC.

(iii) We propose a novel approach of solving the coupled backward Riccati equation (associated with the control gain) and the forward Riccati equation (associated with the estimator gain).
%
%
%
%

\emph{Notation:}
Denote $\mathbb{E}$ as the mathematical expectation operator. $\mathbb{R}^{m}$ presents the $m$-dimensional real Euclidean space.  $tr(X)$ stands for the trace of matrix $X$. Define $\mathcal{F}\{H_k\}$ as the natural filtration generated by the random variable $h_k$, i.e., $\mathcal{F}\{H_k\}=\sigma\{h_0,\ldots,h_k\}$. $X\geq0 (>0)$ denotes that $X$ is a positive semi-definite (positive definite) matrix. $\hat{x}_{k|n}=\mathbb{E}[x_k|\mathcal{F}\{H_k\}]$ denotes the conditional expectation of $x_k$ with respect to $\mathcal{F}\{H_k\}$.

\section{Problem Formulation and Preliminaries}
We shall consider the following discrete-time linear system
\begin{align}
x_{k+1}&=Ax_k+B^1u_k^1+B^2u_k^2+\omega_k,\label{1}\\
y_k^1&=H^1x_k+v_k^1,\label{2}\\
y_k^2&=H^2x_k+v_k^2,\label{3}
\end{align}
where $x_k\in\mathbb{R}^m$ is the state, $u_k^1\in\mathbb{R}^l$ is the control input generated by C1, $u_k^2\in\mathbb{R}^r$ is the control input generated by C2, $y_k^1\in\mathbb{R}^p$ and $y_k^2\in\mathbb{R}^q$ are the observations of sensor 1 and sensor 2. $A,$ $B^1,$ $B^2,$ $H^1$ and $H^2$ are constant matrices with appropriate dimensions. $\omega_k\in\mathbb{R}^m, v_k^1\in\mathbb{R}^p$ and $v_k^2\in\mathbb{R}^q$ are system noise and observation noises with zero mean and covariances $Q_\omega, Q_{v^1}, Q_{v^2}$, respectively. The initial value of state is $x_0$ with mean $\mu$ and covariance $\Sigma$. $x_0,$ $\omega_k,$ $v_k^1$ and $v_k^2$ are Gaussian and independent of each other.

The associated cost function is given by:
\begin{align}
\nonumber J_N=&\mathbb{E}\bigg\{\sum_{k=0}^N[x_k'Qx_k+u_k^{1'}R^1u_k^1+u_{k}^{2'}R^2u_{k}^2]\\
             &\quad\quad+x_{N+1}'\Theta x_{N+1}\bigg\},\label{4}
\end{align}
where $Q,$  $R^1,$ $R^2$ and $\Theta$ are positive semi-definite matrices. $\mathbb{E}$ is the mathematical expectation over the random processes $\{\omega_k\},$  $\{v^1_k\},$ $\{v^2_k\}$ and the random variable $x_0$.
\begin{remark}
The weighting matrices and system matrices can be time-varying which does not affect the derivations of the results of this paper. For simplicity, we consider relevant matrices to be time-invariant.
\end{remark}
Observe from Fig.1 that C1 can obtain the observations of sensor 1 and sensor 2, i.e., $Y^1_k=\{y_0^1,\ldots,y_k^1\}$ and $Y^2_k=\{y_0^2,\ldots,y_k^2\}$. C2 has the access to the observations of sensor 2, i.e., $Y^2_k$. We set $Y_k=\{y_0,\ldots,y_k\}$ as observations of C1 where $y_k=\begin{bmatrix}y_k^{1'}&y_k^{2'}\end{bmatrix}'$. Then the observation equation for C1 can be written as
\begin{align}
y_k&=Hx_k+v_k,\label{14}
\end{align}
where $H=\begin{bmatrix}H^{1'}&H^{2'}\end{bmatrix}'$ and  $v_k=\begin{bmatrix}v^{1'}_k&v^{2'}_k\end{bmatrix}'$. Note that the observation for C2 is (\ref{3}).

Then the problem of this paper is formulated as follows:
\begin{problem}
Find the $\mathcal{F}\{Y_k\}$-measurable $u_k^1$ and the $\mathcal{F}\{Y^2_k\}$-measurable $u^2_k$ such that cost function (\ref{4}) is minimized subject to system (\ref{1}).
\end{problem}

Following a similar discussion in \cite{R19}, we apply Pontryagin's maximum principle to system (\ref{1}) with cost function (\ref{4}) and yield the following costate equations:
\begin{align}
\lambda_{k-1}&=\mathbb{E}[A'\lambda_k+Qx_k|\mathcal{F}\{Y_k\}],\label{5}\\
0&=\mathbb{E}[B^{1'}\lambda_k|\mathcal{F}\{Y_k\}]+R^1u_k^1,\label{6}\\
0&=\mathbb{E}[B^{2'}\lambda_k|\mathcal{F}\{Y^2_k\}]+R^2u^2_k,\label{7}\\
\lambda_{N}&=\mathbb{E}[\Theta x_{N+1}|\mathcal{F}\{Y_{N+1}\}],\label{8}
\end{align}
where $\lambda_k$ is the costate variable and $\Theta$ is the terminal weighting matrix given in (\ref{4}).
\begin{remark}
Since C2 only shares its history control with C1, i.e., the current input generated by C2 is not available to C1, we cannot use the leader-follower approach \cite{R20} to deal with the above control problem. In other words, the method of calculating C2 firstly and then computing C1 based on the results of C2, is not feasible in this structure.
\end{remark}
\begin{remark}
Since the information for C1 and C2 are asymmetric, the method of augmenting the two controllers as one and then applying the traditional optimal control means \cite{R21}, is not applicable. Hence, we aim to develop a novel method of calculating C1 and C2 simultaneously.
\end{remark}
\section{Optimal Control Design}
In this section, we shall present a solution to the optimal control problem stated in Section II.

Firstly, noting the adaptability of $u_1$,  we introduce the following definitions about $u_k^1$:
\begin{align}
\hat{u}_k^1&=\mathbb{E}[u_k^1|\mathcal{F}\{Y^2_k\}], \label{9}\\
\tilde{u}_k^1&=u_k^1-\hat{u}_k^1.\label{10}
\end{align}
It is not hard to find that $\hat{u}^1_k$ and $\tilde{u}^1_k$ have the following properties:
\begin{align}
\label{11}&\mathbb{E}[\tilde{u}^1_k]=0,\mathbb{E}[\tilde{u}^1_k|\mathcal{F}\{Y^2_k\}]=0,\\
\label{12}&\mathbb{E}[\tilde{u}^1_k|\mathcal{F}\{Y_k\}]=\tilde{u}^1_k,\mathbb{E}[\hat{u}^1_k|\mathcal{F}\{Y_k\}]=\hat{u}^1_k.
\end{align}
Through the above definitions, $u_k^1$ is decomposed into two parts, i.e., $\hat{u}^1_k$ and $\tilde{u}^1_k$. Moreover, via the decomposition, $\hat{u}^1_k$ has the same adaptability with $u^2_k$. Then we rewrite system (\ref{1}) and cost function (\ref{4}) as
\begin{align}
x_{k+1}&=Ax_k+Bu_k+B^1\tilde{u}^1_k+\omega_k,\label{13}\\
\nonumber J_N=&\mathbb{E}\bigg\{\sum_{k=0}^N[x_k'Qx_k+u_{k}'Ru_{k}+\tilde{u}_{k}^{1'}R^1\tilde{u}_{k}^1]\\
             &\quad\quad+x_{N+1}'\Theta x_{N+1}\bigg\},\label{15}
\end{align}
where $B=\begin{bmatrix}B^1&B^2\end{bmatrix}$, $u_k=\begin{bmatrix}\hat{u}^1_k\\u^2_k\end{bmatrix}$ and $R=\begin{bmatrix}R^1&0\\0&R^2\end{bmatrix}$.

Next we shall develop the optimal estimators for C1 and C2 respectively in the following lemma.
\begin{lemma}
With observations $\{y_0,\ldots,y_k\}$ from system (\ref{1}) and (\ref{14}), the optimal estimator in device 1 is presented as
\begin{align}
\hat{x}^1_{k|k}&=\hat{x}^1_{k|k-1}+G_{k|k-1}(y_k-H\hat{x}^1_{k|k-1}),\label{16}\\
\hat{x}^1_{k|k-1}&=A\hat{x}^1_{k-1|k-1}+Bu_{k-1}+B^1\tilde{u}_{k-1}^1,\label{new22}
\end{align}
where $G_{k|k-1}=\Sigma^1_{k|k-1}H^{'}(H\Sigma^1_{k|k-1}H^{'}+Q_{v})^{-1}$ and $\Sigma_{k|k-1}^1$ is the estimation error covariance satisfying
\begin{align}
\Sigma^1_{k|k-1}&=\mathbb{E}[(x_k-\hat{x}^1_{k|k-1})(x_k-\hat{x}^1_{k|k-1})']\nonumber\\
&=A\Sigma_{k-1|k-1}^1A'+Q_\omega,\label{17}\\
\nonumber\Sigma^1_{k|k}&=\mathbb{E}[(x_k-\hat{x}^1_{k|k})(x_k-\hat{x}^1_{k|k})']\nonumber\\
\nonumber              &=(I-G_{k|k-1}H)\Sigma_{k|k-1}^1(I-G_{k|k-1}H)'\\
                       &\quad+G_{k|k-1}Q_{v}G^{'}_{k|k-1},\label{18}
\end{align}
with initial value $x^1_{0|-1}=\mu$ and $\Sigma^1_{0|-1}=\Sigma$.

On the other hand, given observations $\{y_0^2,\ldots,y_k^2\}$ from system (\ref{1}) and (\ref{3}), the optimal estimator for device 2 is given by
\begin{align}
\hat{x}^2_{k|k}&=\hat{x}^2_{k|k-1}+G^2_{k|k-1}(y_k^2-H^2\hat{x}^2_{k|k-1}),\label{19}\\
\hat{x}^2_{k|k-1}&=A\hat{x}^2_{k-1|k-1}+Bu_{k-1},\label{new23}
\end{align}
where $G^2_{k|k-1}=\Sigma^2_{k|k-1}H^{2'}(H^2\Sigma^2_{k|k-1}H^{2'}+Q_{v^2})^{-1}$ and $\Sigma_{k|k-1}^2$ is the estimation error covariance satisfying
\begin{align}
\Sigma^2_{k+1|k}&=\mathbb{E}[(x_{k+1}-\hat{x}^2_{k+1|k})(x_{k+1}-\hat{x}^2_{k+1|k})']\nonumber\\
&=A\Sigma_{k|k}^2A'+\mathbb{E}[A(x_{k}-\hat{x}^2_{k|k})\tilde{u}_{k}^{1'}B^{1'}]\nonumber\\
&\quad+\mathbb{E}[B^1\tilde{u}^1_{k}(x_{k}-\hat{x}^2_{k|k})'A']\nonumber\\
&\quad+\mathbb{E}[B^1\tilde{u}^1_{k}\tilde{u}^{1'}_{k}B^{1'}]+Q_\omega,\label{20}\\
\nonumber\Sigma^2_{k|k}&=\mathbb{E}[(x_k-\hat{x}^2_{k|k})(x_k-\hat{x}^2_{k|k})']\nonumber\\
\nonumber              &=(I-G^2_{k|k-1}H^2)\Sigma_{k|k-1}^2(I-G^2_{k|k-1}H^2)'\\
                       &\quad+G^2_{k|k-1}Q_{v^2}G^{2'}_{k|k-1},\label{21}
\end{align}
with initial value $x^2_{0|-1}=\mu$ and $\Sigma^2_{0|-1}=\Sigma$.
\begin{proof}
The above optimal estimators can be obtained directly by using the standard Kalman filtering \cite{R22}.
\end{proof}
\end{lemma}
\begin{remark}
As can be seen from (\ref{20}), the EEC $\Sigma^2_{k+1|k}$ is coupled with the controller $\tilde{u}^1_{k}$ which means that the well-known separation principle fails, resulting in a long-standing fundamental problem of coupled control-estimation. How to decouple the controller from the estimation is a challenging and unsolved problem \cite{R18}.
\end{remark}
It is noted that by making use of (\ref{1}) and the costate equations (\ref{5})-(\ref{8}), the optimal $u^2_k$ can be easily obtained. However, to calculate $u_k^1$, we need the information of $u_k^2$ which is not available due to partial historical sharing. To calculate the two controllers simultaneously, we firstly put forward the following lemma.
\begin{lemma}
The costate equations (\ref{5})-(\ref{8}) are equivalent to the following equations:
\begin{align}
\lambda_{k-1}&=\mathbb{E}[A'\lambda_k+Qx_k|\mathcal{F}\{Y_k\}],\label{22}\\
0&=\mathbb{E}[B'\lambda_k|\mathcal{F}\{Y^2_k\}]+Ru_k,\label{23}\\
\nonumber0&=\mathbb{E}[B^{1'}\lambda_k|\mathcal{F}\{Y_k\}]-\mathbb{E}[B^{1'}\lambda_k|\mathcal{F}\{Y^2_k\}]\\
&\quad+R^1\tilde{u}^1_k,\label{24}\\
\lambda_{N}&=\mathbb{E}[\Theta x_{N+1}|\mathcal{F}\{Y_{N+1}\}].\label{25}
\end{align}
\begin{proof}
Taking the mathematical expectation over (\ref{6}) with $\mathcal{F}\{Y_k^2\}$ and using (\ref{9}), we have that
\begin{align}
0&=\mathbb{E}[B^{1'}\lambda_k|\mathcal{F}\{Y^2_k\}]+R^1\hat{u}^1_k,\label{26}
\end{align}
Augmenting (\ref{7}) with (\ref{26}), it is not hard to obtain equation (\ref{23}). Subtracting (\ref{26}) from (\ref{6}) and using (\ref{10}), (\ref{24}) is readily obtained.
\end{proof}
\end{lemma}
Next we define the following coupled Riccati equations:
\begin{align}
P_k&=A'P_{k+1}A-M_k'\Upsilon_k^{-1}M_k+Q,\label{27}\\
S_k&=A'\Phi_{k+1}A-L_k^{0'}\Lambda_k^{-1}L_k+Q,\label{28}
\end{align}
where
\begin{align}
M_k&=B'P_{k+1}A,\label{29}\\
\Upsilon_k&=B'P_{k+1}B+R,\label{30}\\
\Phi_k&=(P_k-S_k)G^2_{k|k-1}H^2+S_k,\label{31}\\
L_k^0&=B^{1'}\Phi_{k+1}'A,\label{new27}\\
L_k&=B^{1'}\Phi_{k+1}A,\label{32}\\
\Lambda_k&=B^{1'}\Phi_{k+1}B^1+R^1,\label{33}
\end{align}
with terminal values $P_{N+1}=S_{N+1}=\Theta$ and $G^2_{k|k-1}$ defined in Lemma 1.

We now give the optimal controllers in the following theorem.
\begin{theorem}
Assuming that $\Upsilon_k$ and $\Lambda_k$ are invertible for $k=N,\ldots,0$, the optimal controllers for Problem 1 are given by
\begin{align}
u_k&=-\Upsilon_k^{-1}M_k\hat{x}_{k|k}^2\label{35},\\
\tilde{u}^1_k&=-\Lambda_k^{-1}L_k(\hat{x}^1_{k|k}-\hat{x}_{k|k}^2),\label{36}
\end{align}
where $\hat{x}^2_{k|k}$ and $\hat{x}^1_{k|k}$ are defined as in Lemma 1, and $\Upsilon_k, M_k, \Lambda_k, L_k$ are as in (\ref{27})-(\ref{33}). Accordingly, the optimal $u_k^1=\begin{bmatrix}I&0\end{bmatrix}u_k+\tilde{u}^1_k$, and the optimal $u^2_k=\begin{bmatrix}0&I\end{bmatrix}u_k$. The optimal cost function is as
\begin{align}
\nonumber J^*_N&=\mathbb{E}[x_0'P_0\hat{x}^2_{0|0}+x_0'S_0(\hat{x}^1_{0|0}-\hat{x}^2_{0|0})]+tr(\Sigma^1_{N+1|N+1}\Theta)\\
\nonumber&\quad+\hspace{-0.8mm}\sum_{k=0}^Ntr\{\Sigma^1_{k|k}[Q\hspace{-0.8mm}-\hspace{-0.8mm}A'(S_{k+1}\hspace{-0.8mm}-\hspace{-0.8mm}\Phi_{k+1}\hspace{-0.8mm}-S_{k+1}G_{k+1|k}H)\\
&\quad\quad\times A]\hspace{-0.8mm}+\hspace{-0.8mm}[Q_\omega(S_{k+1}G_{k+1|k}H\hspace{-0.8mm}+\hspace{-0.8mm}\Phi_{k+1}\hspace{-0.8mm}-\hspace{-0.8mm}S_{k+1})]\}\label{37}.
\end{align}
Moreover, the solution to the FBSDEs (\ref{13}) and (\ref{22}) is as
\begin{align}
\lambda_{k-1}=P_k\hat{x}^2_{k|k}+S_k(\hat{x}_{k|k}^1-\hat{x}^2_{k|k}).\label{34}
\end{align}
\end{theorem}
\begin{proof}
See Appendix A.
\end{proof}
It should be emphasized that $\Sigma_{k+1|k}^2$ in (\ref{20}) is coupled with $\tilde{u}_k^1$. Noting the special form of (\ref{36}), and the relationship between (\ref{36}) and (\ref{20}), we succeed to decouple $\tilde{u}_k^1$ from $\Sigma_{k+1|k}^2$ in the following lemma.
\begin{lemma}
The estimation error covariance (\ref{20}) can be calculated as
\begin{align}
\nonumber\Sigma^2_{k+1|k}&=(A-B^1\Gamma_k)\big\{[(I\hspace{-0.8mm}-G^2_{k|k-1}H^2)\Sigma_{k|k-1}^2(I\hspace{-0.8mm}-G^2_{k|k-1}\\
                       &\quad\times H^2)'\hspace{-0.8mm}+G^2_{k|k-1}Q_{v^2}G^{2'}_{k|k-1}]-\hspace{-0.8mm}\Sigma^1_{k|k}\big\}(A-B^1\Gamma_k)'\nonumber\\
&\quad+A\Sigma^1_{k|k}A'+Q_\omega,\label{46}
\end{align}
where $\Gamma_k=\Lambda^{-1}_kL_k$, and $\Sigma^1_{k|k}$ can be calculated as in Lemma 1 iteratively.
\end{lemma}
\begin{proof}
It can be seen from (\ref{20}) that the coupled terms are $\mathbb{E}[A(x_{k}-\hat{x}^2_{k|k})\tilde{u}_{k}^{1'}B^{1'}]$,  $\mathbb{E}[B^1\tilde{u}^1_{k}(x_{k}-\hat{x}^2_{k|k})'A']$ and $\mathbb{E}[B^1\tilde{u}^1_{k}\tilde{u}^{1'}_{k}B^{1'}]$.
Firstly, note that
\begin{align}
\nonumber&\mathbb{E}[(\hat{x}^1_{k|k}-\hspace{-0.8mm}\hat{x}^2_{k|k})(x_{k}-\hat{x}^2_{k|k})']\\
\nonumber &=\mathbb{E}\{[(x_k-\hat{x}^2_{k|k})-(x_k-\hat{x}^1_{k|k})](x_{k}-\hat{x}^2_{k|k})'\}\\
\nonumber &=\Sigma^2_{k|k}-\mathbb{E}[(x_k-\hat{x}^1_{k|k})(x_{k}-\hat{x}^2_{k|k})']\\
\nonumber &=\Sigma^2_{k|k}-\mathbb{E}[(x_k-\hat{x}^1_{k|k})x_{k}]'\\
\nonumber &=\Sigma^2_{k|k}-\mathbb{E}\{(x_k-\hat{x}^1_{k|k})[(x_{k}-\hat{x}^1_{k|k})+\hat{x}^1_{k|k}]\}'\\
 &=\Sigma^2_{k|k}-\Sigma^1_{k|k},\label{new16}
\end{align}
and
\begin{align}
\nonumber&\mathbb{E}[(\hat{x}^1_{k|k}-\hspace{-0.8mm}\hat{x}^2_{k|k})(\hat{x}^1_{k|k}-\hspace{-0.8mm}\hat{x}^2_{k|k})']\\
\nonumber&=\mathbb{E}\{[(x_k\hspace{-0.8mm}-\hspace{-0.8mm}\hat{x}^2_{k|k})\hspace{-0.8mm}-\hspace{-0.8mm}(x_k\hspace{-0.8mm}-\hspace{-0.8mm}\hat{x}^1_{k|k})][(x_k-\hat{x}^2_{k|k})\hspace{-0.8mm}-\hspace{-0.8mm}(x_k-\hspace{-0.8mm}\hat{x}^1_{k|k})]'\}\\
\nonumber&=\Sigma^2_{k|k}-\mathbb{E}[(x_k-\hat{x}^2_{k|k})(x_k-\hat{x}^1_{k|k})']\\
\nonumber&\quad-\mathbb{E}[(x_k-\hat{x}^1_{k|k})(x_k-\hat{x}^2_{k|k})']+\Sigma^1_{k|k}\\
\nonumber&=\Sigma^2_{k|k}-\mathbb{E}[x_k(x_k-\hat{x}^1_{k|k})']-\mathbb{E}[(x_k-\hat{x}^1_{k|k})x_k']+\Sigma^1_{k|k}\\
&=\Sigma^2_{k|k}-\Sigma^1_{k|k}.\label{new17}
\end{align}
In virtue of  (\ref{36}), (\ref{new16}) and (\ref{new17}), (\ref{20}) can be written as
\begin{align}
\nonumber\Sigma^2_{k+1|k}&=(A-B^1\Gamma_k)(\Sigma_{k|k}^2-\Sigma^1_{k|k})(A-B^1\Gamma_k)'\nonumber\\
&\quad+A\Sigma^1_{k|k}A'+Q_\omega.\label{48}
\end{align}
Obviously, $\tilde{u}_k^1$ is decoupled from $\Sigma_{k+1|k}^2$. Substituting (\ref{21}) into (\ref{48}), it can be readily obtained that (\ref{46}) holds.
\end{proof}
\begin{remark}
Although $\tilde{u}_k^1$ is decoupled from $\Sigma_{k+1|k}^2$ via Lemma 3, it is noted that the estimation gain  is coupled with the control gain and the control gain is dependent on the estimation gain; see (\ref{46}) and (\ref{31}). Besides, Riccati equation (\ref{46}) for the estimation gain is forward and Riccati equation (\ref{28}) for the control gain is backward which implies that the forward and backward Riccati equations (\ref{46}) and (\ref{28}) cannot be solved simultaneously. In other words, the estimators and controllers cannot be calculated.
\end{remark}
To address this problem, we now consider steady-state solutions for the Kalman filtering and control problem. To this end, it is necessary to make the following assumption:
\begin{assumption}
The solutions of the coupled Riccati equations (\ref{27}) and (\ref{28}) converge to $P$ and $S$ respectively and (\ref{46}) converges to $\Sigma^2$ when $k$ and $N$ are large enough.
\end{assumption}
Under Assumption 1 and ($A,H$) is detectable, it is easy to know that $P$ and $S$ satisfy the following two algebraic Riccati equations:
\begin{align}
 P&=A'PA-M'\Upsilon^{-1}M+Q, \label{new10}\\
 S&=A'\Phi A-L^{0'}\Lambda^{-1}L+Q,\label{new11}
\end{align}
where
\begin{align}
\nonumber M&=B'PA,\\
\nonumber \Upsilon&=B'PB+R,\\
\nonumber \Phi&=(P-S)G^2H^2+S,\\
\nonumber L^0&=B^{1'}\Phi'A,\\
\nonumber L&=B^{1'}\Phi A,\\
\nonumber \Lambda&=B^{1'}\Phi B^1+R^1,
\end{align}
with
\begin{align}
\nonumber G^2&=\Sigma^2H^{2'}(H^2\Sigma^2H^{2'}+Q_{v^2})^{-1},\\
\nonumber\Sigma^2\hspace{-0.8mm}&=\hspace{-0.8mm}(A\hspace{-0.8mm}-\hspace{-0.8mm}B^1\Lambda^{-1} L)\big\{\hspace{-0.8mm}[(I\hspace{-0.8mm}-\hspace{-0.8mm}G^2H^2)\Sigma^2(I\hspace{-0.8mm}-\hspace{-0.8mm}G^2H^2)'\hspace{-0.8mm}+\hspace{-0.8mm}G^2Q_{v^2}G^{2'}]\\
&\quad-\tilde{\Sigma}^1\big\}(A-B^1\Lambda^{-1} L)'+A\tilde{\Sigma}^1A'+Q_\omega,\label{new26}\\
\nonumber\tilde{\Sigma}^1&=(I-GH)\Sigma^1(I-GH)'+GQ_{v}G^{'},\\
\nonumber G&=\Sigma^1H^{'}(H\Sigma^1H^{'}+Q_{v})^{-1},\\
\nonumber\Sigma^1&=A\tilde{\Sigma}^1A'+Q_\omega,
\end{align}
Now we shall calculate $P$ and $S$ in (\ref{new10}) and (\ref{new11}). By applying  (\ref{new10}) and (\ref{new11}), in order to save the number of symbols, we use the same symbols as in (\ref{27}) and (\ref{28}) to define the two forward iterations as:
\begin{align}
P_{k+1}&=A'P_{k}A-M_k'\Upsilon_k^{-1}M_k+Q,\label{new1}\\
S_{k+1}&=A'\Phi_{k}A-L_k^{0'}\Lambda_k^{-1}L_k+Q,\label{new2}
\end{align}
where
\begin{align}
\nonumber M_k&=B'P_{k}A,\\
\nonumber\Upsilon_k&=B'P_{k}B+R,\\
\nonumber\Phi_k&=(P_k-S_k)G^2_{k|k-1}H^2+S_k,\\
\nonumber L_k^0&=B^{1'}\Phi_{k}'A,\\
\nonumber L_k&=B^{1'}\Phi_{k}A,\\
\nonumber\Lambda_k&=B^{1'}\Phi_{k}B^1+R^1,
\end{align}
with
\begin{align}
G^2_{k|k-1}&=\Sigma^2_{k|k-1}H^{2'}(H^2\Sigma^2_{k|k-1}H^{2'}+Q_{v^2})^{-1},\nonumber\\
\nonumber\Sigma^2_{k+1|k}&=(A-B^1\Lambda^{-1}_kL_k)\big\{[(I\hspace{-0.8mm}-G^2_{k|k-1}H^2)\Sigma_{k|k-1}^2(I\hspace{-0.8mm}\\
                       &\quad-G^2_{k|k-1} H^2)'\hspace{-0.8mm}+G^2_{k|k-1}Q_{v^2}G^{2'}_{k|k-1}]-\hspace{-0.8mm}\Sigma^1_{k|k}\big\}\nonumber\\
&\quad\times(A-B^1\Lambda^{-1}_kL_k)'+A\Sigma^1_{k|k}A'+Q_\omega,\label{new50}\\
\nonumber\Sigma^1_{k|k}&=(I-G_{k|k-1}H)\Sigma_{k|k-1}^1(I-G_{k|k-1}H)'\\
\nonumber              &\quad+G_{k|k-1}Q_{v}G^{'}_{k|k-1},\\
\nonumber G_{k|k-1}&=\Sigma^1_{k|k-1}H^{'}(H\Sigma^1_{k|k-1}H^{'}+Q_{v})^{-1},\\
\nonumber\Sigma^1_{k|k-1}&=A\Sigma_{k-1|k-1}^1A'+Q_\omega,
\end{align}
with initial values $P_0=S_0=\delta I$ ($\delta>0$) and  $\Sigma_{0|-1}^1=\Sigma_{0|-1}^2=\sigma$.
It is clear that (\ref{new1}) and (\ref{new2}) are standard forward iterations which are different from (\ref{27}) and (\ref{28}) where the iterations are infeasible. Suppose (\ref{new1}), (\ref{new2}) and (\ref{new50}) are convergent when $k$ is large enough. Then the solutions would be the same as those of (\ref{27}), (\ref{28}) and (\ref{46}).

Now we are in the position to present the optimal controllers in the following theorem.
\begin{theorem}
Suppose that ($A,H$) is detectable, (\ref{new1}), (\ref{new2}) converge to $P$ and $S$, and (\ref{new50}) converges to $\Sigma^2$, where $P\geq0$, $S$ and $\Sigma^2$ have the same values as in (\ref{new10}), (\ref{new11}) and (\ref{new26}),  when the iteration $k$ is large enough. Then the optimal controllers are given as
\begin{align}
u_k&=-\Upsilon^{-1}M\hat{x}_{k|k}^2\label{new8},\\
\tilde{u}^1_k&=-\Lambda^{-1}L(\hat{x}^1_{k|k}-\hat{x}_{k|k}^2).\label{new9}
\end{align}
Furthermore, the above optimal controllers also make the system (\ref{13}) bounded in the mean-square sense.

Accordingly, the steady-state kalman filters are as
\begin{align}
\hat{x}^1_{k|k}&=\hat{x}^1_{k|k-1}+G(y_k-H\hat{x}^1_{k|k-1}),\label{new12}\\
\hat{x}^1_{k|k-1}&=A\hat{x}^1_{k-1|k-1}+Bu_{k-1}+B^L\tilde{u}_{k-1}^L,\label{new24}
\end{align}
and
\begin{align}
\hat{x}^2_{k|k}&=\hat{x}^2_{k|k-1}+G^2(y_k^2-H^2\hat{x}^2_{k|k-1}),\label{new13}\\
\hat{x}^2_{k|k-1}&=A\hat{x}^2_{k-1|k-1}+Bu_{k-1}.\label{new25}
\end{align}
\end{theorem}
\begin{proof}
The proof of Theorem 2 is put into Appendix B.
\end{proof}
\begin{remark}
Although (\ref{new8}) and (\ref{new9}) are suboptimal in finite horizon , under the assumptions in Theorem 2, (\ref{new8}) and (\ref{new9}) are optimal for the infinite-horizon case. In other words, when the iteration $k$ is large enough, (\ref{new8}) and (\ref{new9}) are optimal for the finite-horizon case either.
\end{remark}
\section{Numerical Examples}
Consider the system (\ref{3}), (\ref{14}), (\ref{13}) and the cost function (\ref{15}) with
\begin{align}
\nonumber&A=2.7,B^1=1.2,B=\begin{bmatrix}1.2&1.1\end{bmatrix},\\
\nonumber&H^2=1.1,H=\begin{bmatrix}1.2&1.1\end{bmatrix},Q=\hspace{-0.8mm}R^1=\hspace{-0.8mm}1,R=\begin{bmatrix}1&0\\0&1\end{bmatrix}\\
\nonumber&Q_{v^2}=Q_\omega=1,Q_v=\begin{bmatrix}1&0\\0&1\end{bmatrix},\mu=0,\sigma=1.
\end{align}
The initial values for forward Riccati equations (\ref{new1}), (\ref{new2}) and (\ref{new50}) are given by
\begin{align}
\Sigma_{0|-1}^1=\Sigma_{0|-1}^2=0.1,P_0=S_0=3\nonumber.
\end{align}
Direct calculation yields
\begin{align}
\nonumber &P_0=3,\Upsilon_0=\begin{bmatrix}5.3200&3.9600\\3.9600&4.6300\end{bmatrix},M_0=\begin{bmatrix}4.6800\\4.2900\end{bmatrix},\\
\nonumber &P_1=3.4436,\Upsilon_1=\begin{bmatrix}5.9587&4.5455\\4.5455&5.1667\end{bmatrix},M_1=\begin{bmatrix}11.1572\\10.2274\end{bmatrix},\\
\nonumber &P_2=3.4793,\Upsilon_2=\begin{bmatrix}6.0101&4.5926\\4.5926&5.2099\end{bmatrix},M_2=\begin{bmatrix}11.2728\\10.3334\end{bmatrix},\\
\nonumber &P_3=3.4818,\Upsilon_3=\begin{bmatrix}6.0138&4.5959\\4.5959&5.2129\end{bmatrix},M_3=\begin{bmatrix}11.2809\\10.3409\end{bmatrix},\\
\nonumber&S_0=3,\Lambda_0=5.3200,L_0=4.6800,\Phi_0=3,\\
\nonumber&S_1=5.1109,\Lambda_1=6.5855,L_1=12.5673,\Phi_1=3.8788,\\
\nonumber&S_2=5.2938,\Lambda_2=6.5018,L_2=12.3790,\Phi_2=3.8207,\\
\nonumber&S_3=5.2839,\Lambda_3=6.4046,L_3=12.1603,\Phi_3=3.7532.
\end{align}
By applying the obtained values to Theorem 2, the optimal controllers can be calculated as
\begin{align}
\nonumber &u_0=\begin{bmatrix}-0.1443\\-0.1323\end{bmatrix},u_1=\begin{bmatrix}-2.1685\\-1.9878\end{bmatrix},u_2=\begin{bmatrix}-1.1871\\ -1.0882\end{bmatrix},\\
\nonumber&u_3=\begin{bmatrix}-9.7144\\-8.9049\end{bmatrix},\tilde{u}_0^L=1.5714,\tilde{u}_1^L=0.9316,\\
\nonumber&\tilde{u}_2^L=11.3050,\tilde{u}_3^L=-10.1792.
\end{align}
Due to space limitation, we omit the rest values for $k=4,\ldots,N$. We draw the trajectories of $P_k$ and $S_k$ in Fig. 2. As can be seen from Fig.2, $P_k$ and $S_k$ converge to fixed values. Accordingly, the regulated states are drawn in Fig. 3. From Fig.3, it can be seen that states under the control of (\ref{new8}) and (\ref{new9}) are bounded in the mean square sense and have better performance than the case without the control.

As can be seen from Fig. 4, the two estimators can be asymptotically stable which means that the proposed algorithm in Theorem 2 also makes estimators stable. Besides, estimator 1 has the better performance than estimator 2 and the estimator 1 converges faster than estimator 2. From Fig. 5, estimated values of estimator 1 is closer to the true values than those of estimator 2.
\begin{figure}[htbp]
  \begin{center}
  \includegraphics[width=0.41\textwidth]{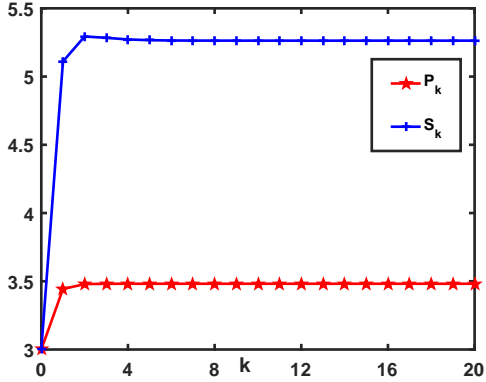}
  \caption{Trajectories of $P_k$ and $S_k$.} \label{fig:digit}
  \end{center}
\end{figure}
\begin{figure}[htbp]
  \begin{center}
  \includegraphics[width=0.41\textwidth]{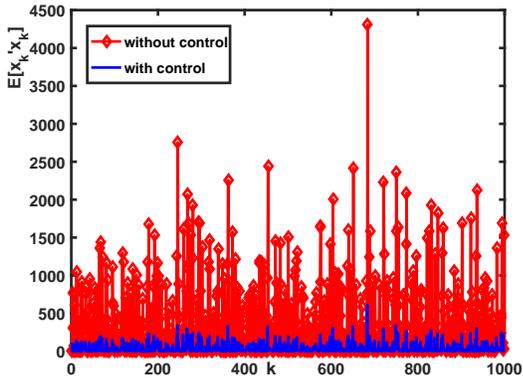}
  \caption{Comparisons of states with control and without control.} \label{fig:digit}
  \end{center}
\end{figure}
\begin{figure}[htbp]
  \begin{center}
  \includegraphics[width=0.41\textwidth]{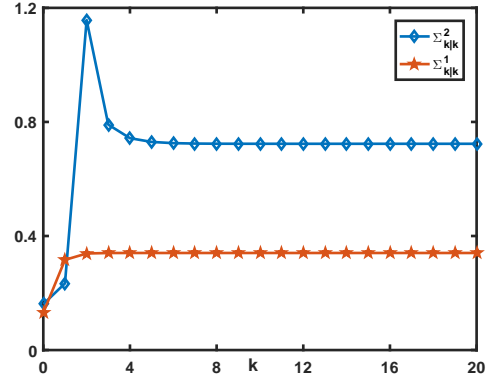}
  \caption{The performance of estimation error covariances $\Sigma_{k|k}^2$ and $\Sigma_{k|k}^1$.} \label{fig:digit}
  \end{center}
\end{figure}
\begin{figure}[htbp]
  \begin{center}
  \includegraphics[width=0.41\textwidth]{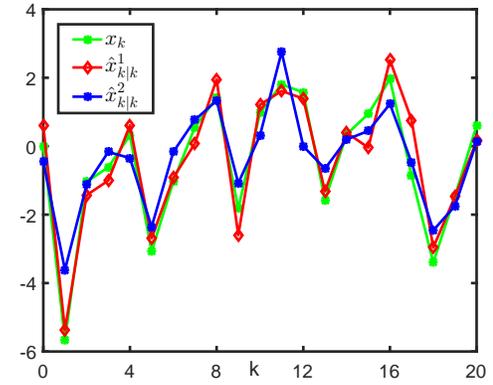}
  \caption{The comparisons of $x_k$, $\hat{x}_{k|k}^1$ and $\hat{x}_{k|k}^2$.} \label{fig:digit}
  \end{center}
\end{figure}
\section{Conclusion}
This paper investigates the optimal decentralized control for NCSs with asymmetric information and partial history sharing, where C2 shares its observations and partial history control with C1 and C1 does not share information with C2. Based on this NCSs model, the optimal estimators for C1 and C2 are presented by using asymmetric observations. It is noted that the EEC is coupled with the controller. Through Pontryagin's maximum principle, the solution to the FBSDEs is given. By making use of this solution, the optimal controllers are shown in terms of two coupled Riccati equations. Combining the optimal controllers with the EEC, the controller is decoupled from the EEC. It should be stressed that the control gain is dependent on the estimation gain. Furthermore, the control gain satisfies the backward Riccati equation and the estimation gain is related to the forward Riccati equation. We propose a iterative method to compute solutions to the coupled forward and backward Riccati equations. Numerical examples are given to show the effectiveness of the proposed algorithm and the stability of estimators.
\appendices
\section{Proof of Theorem 1}
\begin{proof}
Suppose that $\Upsilon_k$ and $\Lambda_k$ are invertible for $k=N,\ldots,0$, We shall show by mathematical induction that the optimal controllers are as (\ref{35}), (\ref{36}) and $\lambda_{k-1}$ has the form of (\ref{34}). From (\ref{25}) and $P_{N+1}=S_{N+1}=\Theta$, it is readily obtained that (\ref{34}) is valid for $k=N+1$.

For $k=N$, using (\ref{13}), (\ref{25}) and (\ref{11}), we have (\ref{23}) as
\begin{align}
0&=\mathbb{E}[B'\Theta x_{N+1}|\mathcal{F}\{Y^2_{N}\}]+Ru_N\nonumber\\
 &=B'\Theta(A\hat{x}^2_{N|N}+Bu_N)+Ru_N\nonumber.
\end{align}
Then with (\ref{29}) and (\ref{30}), the optimal $u_N$ is as
\begin{align}
u_N=-\Upsilon_N^{-1}M_N\hat{x}^2_{N|N},\label{38}
\end{align}
which means that (\ref{35}) holds for $k=N$. By making use of (\ref{13}), (\ref{25}) and (\ref{12}), (\ref{24}) becomes
\begin{align}
\nonumber0&=\mathbb{E}[B^{1'}\Theta x_{N+1}|\mathcal{F}\{Y_{N}\}]-\mathbb{E}[B^{1'}\Theta x_{N+1}|\mathcal{F}\{Y^2_{N}\}]\\
\nonumber&\quad+R^1\tilde{u}^1_N\\
\nonumber&=B^{1'}\Theta A(\hat{x}^1_{N|N}-\hat{x}^2_{N|N})+(B^{1'}\Theta B^1+R^1)\tilde{u}^1_N.
\end{align}
Thus, by virtue of (\ref{32}) and (\ref{33}), the optimal $\tilde{u}^1_N$ is as
\begin{align}
\tilde{u}^1_N=-\Lambda^{-1}_NL_N(\hat{x}^1_{N|N}-\hat{x}^2_{N|N}),\label{39}
\end{align}
which implies that (\ref{36}) is valid for $k=N$.
Using (\ref{13}), (\ref{38}), (\ref{39}) and (\ref{12}), (\ref{22}) becomes
\begin{align}
\lambda_{N}&=\mathbb{E}[A'\Theta x_{N+1}+Qx_N|\mathcal{F}\{Y_{N}\}]\nonumber\\
           &=A'\Theta(A\hat{x}^1_{N|N}+Bu_N+B^1\tilde{u}^1_N)+Q\hat{x}^1_{N|N}\nonumber\\
           &=(A'\Theta A-M_N'\Upsilon_N^{-1}M_N+Q)\hat{x}^2_{N|N}\nonumber\\
           &\quad+(A'\Theta A-L_N'\Lambda_N^{-1}L_N+Q)(\hat{x}^1_{N|N}-\hat{x}^2_{N|N})\nonumber
\end{align}
With (\ref{27}) and (\ref{28}), it can be obtained that (\ref{34}) holds for $k=N$.

Next following the proof of the mathematical induction, we choose any $l$ for $0\leq l\leq N$. Suppose that $\lambda_{k-1}$, $u_k$ and $\tilde{u}_k^1$ take forms of (\ref{34}), (\ref{35}) and (\ref{36}) respectively for all $k\geq l+1$. Now we will prove that (\ref{34}), (\ref{35}) and (\ref{36}) are valid for $k=l$. Firstly, we shall make the following preparatory work.

Noting (\ref{19}), (\ref{11}) and (\ref{13}), we have that
\begin{align}
\nonumber\hat{x}^2_{k+1|k+1}&=A\hat{x}^2_{k|k}+Bu_{k}+G^{2}_{k+1|k}[H^2(Ax_k+B^1\tilde{u}^1_k\\
&\quad+\omega_k-A\hat{x}^2_{k|k})+v^2_{k+1}].\label{40}
\end{align}
Using (\ref{16}), (\ref{19}), (\ref{11}), (\ref{12}), (\ref{13}) and (\ref{40}), it yields that
\begin{align}
\nonumber &\hat{x}^1_{k+1|k+1}-\hat{x}^2_{k+1|k+1}\\
          &=A(\hat{x}^1_{k|k}\hspace{-0.8mm}-\hspace{-0.8mm}\hat{x}^2_{k|k})\hspace{-0.8mm}+\hspace{-0.8mm}B^1\tilde{u}^1_k\hspace{-0.8mm}+\hspace{-0.8mm}G_{k+1|k}[H(Ax_k\hspace{-0.8mm}+\hspace{-0.8mm}\omega_k\hspace{-0.8mm}-\hspace{-0.8mm}A\hat{x}^1_{k|k})\hspace{-0.8mm}\nonumber\\
          &\quad+\hspace{-0.8mm}v_{k+1}]\hspace{-0.8mm}-\hspace{-0.8mm}G^2_{k+1|k}[H^2(Ax_k\hspace{-0.8mm}+\hspace{-0.8mm}B^1\tilde{u}^1_k\hspace{-0.8mm}+\hspace{-0.8mm}\omega_k\hspace{-0.8mm}-\hspace{-0.8mm}A\hat{x}^2_{k|k})+v^2_{k+1}].\label{41}
\end{align}
Since (\ref{34}) holds for $k\geq l+1$, for $k=l+1$, we have
\begin{align}
\lambda_{l}=P_{l+1}\hat{x}^2_{l+1|l+1}+S_{l+1}(\hat{x}^1_{l+1|l+1}-\hat{x}^2_{l+1|l+1}).\label{42}
\end{align}

Using (\ref{40}), (\ref{41}), (\ref{42}) and (\ref{11}), (\ref{23}) becomes
\begin{align}
0&=\mathbb{E}\{B'\lambda_{l}|\mathcal{F}\{Y^2_l\}\}+Ru_l\nonumber\\
\nonumber&=B'P_{l+1}A\hat{x}^2_{l|l}+(B'P_{l+1}B+R)u_l.
\end{align}
With (\ref{29}) and (\ref{30}), the optimal $u_l$ is as
\begin{align}
u_l=-\Upsilon_l^{-1}M_l\hat{x}^2_{l|l}.\label{43}
\end{align}
Thus, (\ref{35}) is valid for $k=l$. In virtue of (\ref{40}), (\ref{41}), (\ref{42}), (\ref{11}) and (\ref{12}), (\ref{24}) can be calculated as
\begin{align}
\nonumber0&=\mathbb{E}[B^{1'}\lambda_l|\mathcal{F}\{Y_l\}]-\mathbb{E}[B^{1'}\lambda_l|\mathcal{F}\{Y^2_l\}]+R^1\tilde{u}^1_l\\
\nonumber &=B^{1'}P_{l+1}[A\hat{x}^2_{l|l}+Bu_l+G^2_{l+1|l}H^2(A\hat{x}^1_{l|l}+B^1\tilde{u}^1_l\\
\nonumber&\quad-A\hat{x}^2_{l|l})]+B^{1'}S_{l+1}\{A(\hat{x}^1_{l|l}-\hat{x}^2_{l|l})+B^1\tilde{u}^1_l\\
\nonumber&\quad+G_{l+1|l}H(A\hat{x}^1_{l|l}-A\hat{x}^1_{l|l})\\
\nonumber&\quad-G^2_{l+1|l}H^2(A\hat{x}^1_{l|l}+\hspace{-0.8mm}B^1\tilde{u}^1_l-A\hat{x}^2_{l|l})\}\hspace{-0.8mm}-\hspace{-0.8mm}B^{1'}P_{l+1}\{A\hat{x}^2_{l|l}\\
\nonumber&\quad+Bu_l+G^2_{l+1|l}H^2(A\hat{x}^2_{l|l}-A\hat{x}^2_{l|l})\}+R^1\tilde{u}_l^1\\
\nonumber&=\{B^{1'}[(P_{l+1}-S_{l+1})G^2_{l+1|l}H^2+S_{l+1}]B^1+R^1\}\tilde{u}^1_{l}\\
\nonumber&\quad+B^{1'}[(P_{l+1}-S_{l+1})G^2_{l+1|l}H^2+S_{l+1}]A(\hat{x}^1_{l|l}-\hat{x}^2_{l|l}).
\end{align}
With (\ref{31})-(\ref{33}), the optimal $\tilde{u}^1_l$ is as
\begin{align}
\tilde{u}^1_l&=-\Lambda_l^{-1}L_l(\hat{x}^1_{l|l}-\hat{x}_{l|l}^2),\label{44}
\end{align}
which means that (\ref{36}) stands for $k=l$.
Now we shall prove that (\ref{34}) holds for $k=l$. By making use of (\ref{40}), (\ref{41}), (\ref{42}), (\ref{12}), (\ref{43}) and (\ref{44}), (\ref{22}) becomes
\begin{align}
\nonumber \lambda_{l-1}&=\mathbb{E}[A'\lambda_l+Qx_l|\mathcal{F}\{Y_l\}]\\
\nonumber              &=A'P_{l+1}[A\hat{x}^2_{l|l}\hspace{-0.8mm}+\hspace{-0.8mm}Bu_l\hspace{-0.8mm}+\hspace{-0.8mm}G^2_{l+1|l}H^2(A\hat{x}^1_{l|l}\hspace{-0.8mm}+B^1\tilde{u}^1_l\\
\nonumber&\quad-A\hat{x}^2_{l|l})]+A'S_{l+1}\{[A(\hat{x}^1_{l|l}-\hat{x}^2_{l|l})+B^1\tilde{u}^1_l\\
\nonumber&\quad+G_{l+1|l}H(A\hat{x}^1_{l|l}-A\hat{x}^1_{l|l})]-G^2_{l+1|l}H^2(A\hat{x}^1_{l|l}\\
\nonumber&\quad+B^1\tilde{u}^1_l-A\hat{x}^2_{l|l})\}+Q\hat{x}^1_{l|l}\\
\nonumber&=(A'P_{l+1}A-A'P_{l+1}B\Upsilon_l^{-1}M_l+Q)\hat{x}^2_{l|l}\\
\nonumber&\quad+\{A'[(P_{l+1}-S_{l+1})G^2_{l+1|l}H^2+S_{l+1}]A-A'[(P_{l+1}\\
\nonumber&\quad-S_{l+1})G^2_{l+1|l}H^2+S_{l+1}]B^1\Lambda^{-1}_lL_l+Q\}(\hat{x}^1_{l|l}-\hat{x}^2_{l|l}).
\end{align}
In virtue of (\ref{27}), (\ref{28}), (\ref{29}) and (\ref{31}), it can be obtained that (\ref{34}) holds for $k=l$.

 Next we shall show that the optimal cost is as (\ref{37}). To this end, by making use of the solution (\ref{34}) to the FBSDEs, we define the following value function
\begin{align}
V_k=\mathbb{E}[x_k'P_k\hat{x}^2_{k|k}+x_k'S_k(\hat{x}^1_{k|k}-\hat{x}^2_{k|k})].\label{45}
\end{align}
Before proceeding the following proof, by applying the orthogonality principle, we shall firstly give some preliminary work:
\begin{align}
\nonumber&\mathbb{E}[(x_k-\hat{x}_{k|k}^2)'\Pi(x_k-\hat{x}_{k|k})]\\
\nonumber&=\mathbb{E}\{[(x_k-\hat{x}_{k|k}^1)+(\hat{x}_{k|k}^1-\hat{x}_{k|k}^2)]'\Pi[(x_k-\hat{x}_{k|k}^1)+(\hat{x}_{k|k}^1\\
\nonumber&\quad\quad-\hat{x}_{k|k}^2)]\}\\
&=tr[\Sigma_{k|k}^1\Pi]+\mathbb{E}[(\hat{x}_{k|k}^1-\hat{x}_{k|k}^2)'\Pi(\hat{x}_{k|k}^1-\hat{x}_{k|k}^2)],\label{49}
\end{align}
where $\Pi$ is a known matrix with appropriate dimension.
\begin{align}
\nonumber&\mathbb{E}[x_k'\Pi(\hat{x}_{k|k}^1-\hat{x}_{k|k}^2)]\\
\nonumber&=\mathbb{E}\{[(x_k-\hat{x}_{k|k}^1)+(\hat{x}_{k|k}^1-\hat{x}_{k|k}^2)+\hat{x}_{k|k}^2]'\Pi(\hat{x}_{k|k}^1-\hat{x}_{k|k}^2)\\
\nonumber&=\mathbb{E}[(x_k-\hat{x}_{k|k}^1)'\Pi(\hat{x}_{k|k}^1-\hat{x}_{k|k}^2)]\\
\nonumber&\quad+\mathbb{E}[(\hat{x}_{k|k}^1-\hat{x}_{k|k}^2)'\Pi(\hat{x}_{k|k}^1-\hat{x}_{k|k}^2)]\\
         &=\mathbb{E}[(\hat{x}_{k|k}^1-\hat{x}_{k|k}^2)'\Pi(\hat{x}_{k|k}^1-\hat{x}_{k|k}^2)].\label{50}
\end{align}
In virtue of (\ref{49}), we have
\begin{align}
\nonumber&\mathbb{E}(x_k'\Pi\hat{x}_{k|k}^2)\\
\nonumber&=\mathbb{E}\{x_k'\Pi[(\hat{x}_{k|k}^2-x_k)+x_k]\}\\
\nonumber&=\mathbb{E}\{[(x_k-\hat{x}_{k|k}^2)+\hat{x}_{k|k}^2]'\Pi(x_{k|k}^2-x_k)\}+\mathbb{E}(x_k'\Pi x_k)\\
         &=\mathbb{E}(x_k'\Pi x_k)\hspace{-0.8mm}-\hspace{-0.8mm}tr[\Sigma_{k|k}^1\Pi]\hspace{-0.8mm}-\hspace{-0.8mm}\mathbb{E}[(\hat{x}_{k|k}^1\hspace{-0.8mm}-\hspace{-0.8mm}\hat{x}_{k|k}^2)'\Pi(\hat{x}_{k|k}^1-\hat{x}_{k|k}^2)].\label{51}
\end{align}
Now we shall show that the optimal cost is as (\ref{37}). Using (\ref{13}), (\ref{40}), (\ref{41}), (\ref{49})-(\ref{51}) and (\ref{27})-(\ref{33}), we get
\begin{align}
\nonumber &V_k-V_{k+1}\\
\nonumber&=\mathbb{E}[x_k'P_kx_k+(\hat{x}^1_{k|k}-\hat{x}^2_{k|k})'(S_k-P_k)(\hat{x}^1_{k|k}-\hat{x}^2_{k|k})\\
\nonumber&\quad-tr(\Sigma_{k|k}^1P_k)]-\mathbb{E}[x_{k+1}'P_{k+1}\hat{x}^2_{k+1|k+1}\\
\nonumber&\quad+x_{k+1}'S_{k+1}(\hat{x}^1_{k+1|k+1}-\hat{x}^2_{k+1|k+1})]\\
\nonumber&=\mathbb{E}[x_k'P_kx_k+(\hat{x}^1_{k|k}-\hat{x}^2_{k|k})'(S_k-P_k)(\hat{x}^1_{k|k}-\hat{x}^2_{k|k})\\
\nonumber&\quad-tr(\Sigma_{k|k}^1P_k)]-\hspace{-0.8mm}\mathbb{E}\{x_k'A'P_{k+1}A\hat{x}^2_{k|k}\hspace{-0.8mm}+\hspace{-0.8mm}x_k'A'P_{k+1}Bu_k\\
\nonumber&\quad+x_k'A'P_{k+1}G^2_{k+1|k}H^2A(x_k-\hat{x}^2_{k|k})+x_k'A'P_{k+1}G^{2}_{k+1|k}\\
\nonumber&\quad\times H^2B^1\tilde{u}^1_k\hspace{-0.8mm}+\hspace{-0.8mm}x_k'A'S_{k+1}A(\hat{x}^1_{k|k}\hspace{-0.8mm}-\hspace{-0.8mm}\hat{x}^2_{k|k})\hspace{-0.8mm}+\hspace{-0.8mm}x_k'A'S_{k+1}B^1\tilde{u}^1_k\\
\nonumber&\quad+x_k'A'S_{k+1}G_{k+1|k}HA(x_k\hspace{-0.8mm}-\hspace{-0.8mm}\hat{x}^1_{k|k})\hspace{-0.8mm}-\hspace{-0.8mm}x_k'A'S_{k+1}G^2_{k+1|k}H^2\\
\nonumber&\quad\times A(x_k\hspace{-0.8mm}-\hspace{-0.8mm}\hat{x}^2_{k|k})\hspace{-0.8mm}-\hspace{-0.8mm}x_k'A'S_{k+1}G^2_{k+1|k}H^2B^1\tilde{u}^1_k+u_k'B'P_{k+1}\\
\nonumber&\quad\times A\hat{x}^2_{k|k}+u_k'B'P_{k+1}Bu_k\hspace{-0.8mm}+\hspace{-0.8mm}\tilde{u}^{1'}_kB^{1'}P_{k+1}G^2_{k+1|k}H^2B^1\tilde{u}^1_k\\
\nonumber&\quad+\tilde{u}^{1'}_kB^{1'}P_{k+1}G^2_{k+1|k}H^2A(x_k-\hat{x}^2_{k|k})+\tilde{u}_k^{1'}B^{1'}S_{k+1}A\\
\nonumber&\quad\times(\hat{x}^1_{k|k}\hspace{-0.8mm}-\hspace{-0.8mm}\hat{x}^2_{k|k})\hspace{-0.8mm}+\hspace{-0.8mm}\tilde{u}^{1'}_kB^{1'}S_{k+1}B^1\tilde{u}^1_k\hspace{-0.8mm}-\hspace{-0.8mm}\tilde{u}^{1'}_kB^{1'}S_{k+1}G^2_{k+1|k}\\
\nonumber&\quad\times H^2A(x_k-\hat{x}^2_{k|k})-\tilde{u}_k^{1'}B^{1'}S_{k+1}G^2_{k+1|k}H^2B^1\tilde{u}^1_k\\
\nonumber&\quad+\omega_k'[(P_{k+1}-S_{k+1})G^2_{k+1|k}H^2+S_{k+1}G_{k+1}H]\omega_k\}\\
\nonumber&=\mathbb{E}[x_k'P_kx_k+(\hat{x}^1_{k|k}-\hat{x}^2_{k|k})'(S_k-P_k)(\hat{x}^1_{k|k}-\hat{x}^2_{k|k})\\
\nonumber&\quad-tr(\Sigma_{k|k}^1\Pi)]-\mathbb{E}\{x_k'A'P_{k+1}Ax_k-(x_k-\hat{x}^2_{k|k})'A'\\
\nonumber&\quad\times P_{k+1}A(x_k-\hat{x}^2_{k|k})+\hspace{-0.8mm}\hat{x}^2_{k|k}A'P_{k+1}Bu_k\hspace{-0.8mm}+(x_k-\hspace{-0.8mm}\hat{x}^2_{k|k})'A'\\
\nonumber&\quad\times P_{k+1}G^2_{k+1|k}H^2A(x_k-\hat{x}^2_{k|k})+(\hat{x}^1_{k|k}-\hat{x}^2_{k|k})'A'P_{k+1}\\
\nonumber&\quad\times G^2_{k+1|k}H^2B^1\tilde{u}^1_{k}+(\hat{x}^1_{k|k}-\hspace{-0.8mm}\hat{x}^2_{k|k})'A'S_{k+1}A(\hat{x}^1_{k|k}-\hspace{-0.8mm}\hat{x}^2_{k|k})\\
\nonumber&\quad+(\hat{x}^1_{k|k}\hspace{-0.8mm}-\hspace{-0.8mm}\hat{x}^2_{k|k})'A'S_{k+1}B^1\tilde{u}^1_k\hspace{-0.8mm}+\hspace{-0.8mm}(x_k\hspace{-0.8mm}-\hspace{-0.8mm}\hat{x}^1_{k|k})'A'S_{k+1}G_{k+1|k}\\
\nonumber&\quad\times HA(x_k-\hat{x}^1_{k|k})-(x_k-\hat{x}^2_{k|k})'A'S_{k+1}G^2_{k+1|k}H^2A\\
\nonumber&\quad\times(x_k-\hat{x}^2_{k|k})-(\hat{x}^1_{k|k}-\hspace{-0.8mm}\hat{x}^2_{k|k})'A'S_{k+1}G^2_{k+1|k}H^2B^1\tilde{u}^1_k\\
\nonumber&\quad+u_k'B'P_{k+1}A\hat{x}^2_{k|k}\hspace{-0.8mm}+\hspace{-0.8mm}u_k'B'P_{k+1}Bu_k\hspace{-0.8mm}+\hspace{-0.8mm}\tilde{u}^{1'}_kB^{1'}P_{k+1}G^2_{k+1|k}\\
\nonumber&\quad\times H^2B^1\tilde{u}^1_k+\tilde{u}^{1'}_kB^{1'}P_{k+1}G^2_{k+1|k}H^2A(\hat{x}^1_{k|k}-\hspace{-0.8mm}\hat{x}^2_{k|k})\\
\nonumber&\quad+\tilde{u}^{1'}_kB^{1'}S_{k+1}A(\hat{x}^1_{k|k}-\hspace{-0.8mm}\hat{x}^2_{k|k})+\tilde{u}^{1'}_kB^{1'}S_{k+1}B^1\tilde{u}^1_k\\
\nonumber&\quad-\tilde{u}^{1'}_kB^{1'}S_{k+1}G^2_{k+1|k}H^2A(\hat{x}^1_{k|k}-\hspace{-0.8mm}\hat{x}^2_{k|k})-\tilde{u}^{1'}_kB^{1'}S_{k+1}\\
\nonumber&\quad\times G^2_{k+1|k}H^2B^1\tilde{u}^1_k+\omega_k'[(P_{k+1}-S_{k+1})G^2_{k+1|k}H^2\\
\nonumber&\quad+S_{k+1}G_{k+1}H]\omega_k\}\\
\nonumber&=\mathbb{E}\bigg\{x_k'(P_k-A'P_{k+1}A+M_k'\Upsilon_k^{-1}M_k)x_k-[2u_k'B'\\
\nonumber&\quad\times P_{k+1}A\hat{x}^2_{k|k}+u_k'(\Upsilon_k-R)u_k+\hat{x}^{2'}_{k|k}M_k'\Upsilon_k^{-1}M_k\hat{x}^{2}_{k|k}]\\
\nonumber&\quad-\tilde{u}^{1'}_kB^{1'}[(P_{k+1}-S_{k+1})G^2_{k+1|k}H^2+S_{k+1}]A(\hat{x}^1_{k|k}
\end{align}
\begin{align}
\nonumber&\quad-\hat{x}^2_{k|k})-(\hat{x}^1_{k|k}-\hat{x}^2_{k|k})'A'[(P_{k+1}-S_{k+1})G^2_{k+1|k}H^2\\
\nonumber&\quad+\hspace{-0.8mm}S_{k+1}]B^1\tilde{u}^1_k\hspace{-0.8mm}-\hspace{-0.8mm}\tilde{u}^{1'}_kB^{1'}(\Lambda_k\hspace{-0.8mm}-\hspace{-0.8mm}R^1)B^1\tilde{u}^1_k\hspace{-0.8mm}-\hspace{-0.8mm}(\hat{x}^1_{k|k}\hspace{-0.8mm}-\hspace{-0.8mm}\hat{x}^2_{k|k})'\{P_k\\
\nonumber&\quad-A'P_{k+1}A+A[(P_{k+1}-S_{k+1})G^2_{k+1|k}H^2+S_{k+1}]A\\
\nonumber&\quad-S_k+M_k'\Upsilon_k^{-1}M_k\}(\hat{x}^1_{k|k}-\hspace{-0.8mm}\hat{x}^2_{k|k})-(x_k-\hat{x}^1_{k|k})'[Q-A'\\
\nonumber&\quad\times\hspace{-0.8mm}(S_{k+1}G^2_{k+1|k}H^2\hspace{-0.8mm}-\hspace{-0.8mm}P_{k+1}G^2_{k+1|k}H^2\hspace{-0.8mm}-\hspace{-0.8mm}S_{k+1}G_{k+1|k}H)A]\\
\nonumber&\quad\times(x_k-\hat{x}^1_{k|k})-\omega_k'[(P_{k+1}-S_{k+1})G^2_{k+1|k}H^2\\
\nonumber&\quad+S_{k+1}G_{k+1}H]\omega_k\bigg\}\\
\nonumber&=\mathbb{E}\{x_k'Qx_k+u_{k}'Ru_{k}+\tilde{u}_{k}^{1'}R^1\tilde{u}_{k}^1-(u_k+\Upsilon_k^{-1}M_k\hat{x}^2_{k|k})'\\
\nonumber&\quad\times\Upsilon_k(u_k+\Upsilon_k^{-1}M_k\hat{x}^2_{k|k})-[\tilde{u}^1_k+(\Lambda_k^{-1})'L_k^0(\hat{x}^1_{k|k}-\hspace{-0.8mm}\hat{x}^2_{k|k})]'\\
\nonumber&\quad\times\Lambda_k[\tilde{u}^1_k+\hspace{-0.8mm}\Lambda_k^{-1}L_k(\hat{x}^1_{k|k}-\hspace{-0.8mm}\hat{x}^2_{k|k})]\}\hspace{-0.8mm}-tr\{\Sigma^1_{k|k}[Q-\hspace{-0.8mm}A'(S_{k+1}\\
\nonumber&\quad-\Phi_{k+1}-S_{k+1}G_{k+1|k}H)A]\}-tr[Q_\omega(S_{k+1}G_{k+1}H\\
\nonumber&\quad+\Phi_{k+1}-S_{k+1})].
\end{align}
Adding from $k=0$ to $k=N$ on both sides of the above equation, the cost function (\ref{15}) can be calculated as
\begin{align}
\nonumber J_N&=\mathbb{E}[x_0'P_0\hat{x}^2_{0|0}+x_0'S_0(\hat{x}^1_{0|0}-\hat{x}^2_{0|0})]+tr(\Sigma^1_{N+1|N+1}Q)\\
\nonumber&\quad+\sum_{k=0}^N\{(u_k+\Upsilon_k^{-1}M_k\hat{x}^2_{k|k})'\Upsilon_k(u_k+\Upsilon_k^{-1}M_k\hat{x}^2_{k|k})\\
\nonumber&\quad+[\tilde{u}^1_k+(\Lambda_k^{-1})'L_k^0(\hat{x}^1_{k|k}-\hspace{-0.8mm}\hat{x}^2_{k|k})]'\Lambda_k[\tilde{u}^1_k+\hspace{-0.8mm}\Lambda_k^{-1}L_k(\hat{x}^1_{k|k}\\
\nonumber&\quad-\hspace{-0.8mm}\hat{x}^2_{k|k})]\}+\sum_{k=0}^N\{tr\{\Sigma^1_{k|k}[Q-\hspace{-0.8mm}A'(S_{k+1}-\Phi_{k+1}-S_{k+1}\\
\nonumber&\quad\hspace{-0.8mm}\times \hspace{-0.8mm} G_{k+1|k}H)A]\}\hspace{-0.8mm}+\hspace{-0.8mm}tr[Q_\omega(S_{k+1}G_{k+1}H\hspace{-0.8mm}+\hspace{-0.8mm}\Phi_{k+1}\hspace{-0.8mm}-\hspace{-0.8mm}S_{k+1})]\}
\end{align}
Substituting the optimal controllers (\ref{35}) and (\ref{36}) into the above equation, we have the optimal cost  as (\ref{37}). This ends the proof of Theorem 1.
\end{proof}
\section{Proof of Theorem 2}
\begin{proof}
Assuming that ($A,H$) is detectable, (\ref{new1}), (\ref{new2}) converge to $P$, $S$ and (\ref{new50}) converges to $\Sigma^2$ when the iteration $k$ is large enough, the optimal controllers (\ref{new8}) and (\ref{new9}) can be obtained by the similar procedure in Theorem 1. Under the assumption that ($A,H$) is detectable, the stead-state kalman filter (\ref{new12}) can be derived \cite{R22}.

Now we shall show that (\ref{new8}) and (\ref{new9}) make the system (\ref{13}) bounded in the mean-square sense. Substituting (\ref{new8}) and (\ref{new9}) into the system (\ref{13}), we have
\begin{align}
\nonumber x_{k+1}&=Ax_k-B\Upsilon^{-1}M\hat{x}_{k|k}^2-B^1\Lambda^{-1}L(\hat{x}_{k|k}^1-\hat{x}_{k|k}^2)+\omega_k\\
\nonumber        &=Ax_k+B\Upsilon^{-1}M[(x_k-\hat{x}_{k|k}^2)-x_k]\\
\nonumber        &\quad-B^1\Lambda^{-1}L(\hat{x}_{k|k}^1-\hat{x}_{k|k}^2)+\omega_k\\
\nonumber        &=(A-B\Upsilon^{-1}M)x_k+B\Upsilon^{-1}M(x_k-\hat{x}_{k|k}^2)\\
                 &\quad-B^1\Lambda^{-1}L(\hat{x}_{k|k}^1-\hat{x}_{k|k}^2)+\omega_k.\label{new14}
\end{align}
Next we give the following property:
\begin{align}
\nonumber&\mathbb{E}[x_k'(\hat{x}_{k|k}^1-\hat{x}_{k|k}^2)]\\
\nonumber&=\mathbb{E}\{[(x_k-\hat{x}_{k|k}^1)+(\hat{x}_{k|k}^1-\hat{x}_{k|k}^2)+\hat{x}_{k|k}^2]'(\hat{x}_{k|k}^1-\hat{x}_{k|k}^2)\}\\
\nonumber&=tr(\Sigma_{k|k}^1+\Sigma_{k|k}^2-\Sigma_{k|k}^1)\\
         &=tr(\Sigma_{k|k}^2).\label{new15}
\end{align}
By making use of (\ref{new16}), (\ref{new17}), (\ref{new15}) and (\ref{new14}), we get
\begin{align}
\nonumber&\mathbb{E}(x_{k+1}'x_{k+1})\\
\nonumber&=\mathbb{E}[x_k'(A\hspace{-0.8mm}-\hspace{-0.8mm}B\Upsilon^{-1}M)'(A\hspace{-0.8mm}-\hspace{-0.8mm}B\Upsilon^{-1}M)x_k]+\hspace{-0.8mm}2tr[\Sigma_{k|k}^2(A-B\\
\nonumber&\quad\times\Upsilon^{-1}M)B\Upsilon^{-1}M-2\Sigma_{k|k}^2(A-B\Upsilon^{-1}M)'B^1\Lambda^{-1}L\\
\nonumber&\quad-2(\Sigma_{k|k}^2-\Sigma_{k|k}^1)M'\Upsilon^{-1}B'B^1\Lambda^{-1}L+\Sigma_{k|k}^2M'\Upsilon^{-1}B'\\
&\quad\times B\Upsilon^{-1}M+(\Sigma_{k|k}^2-\Sigma_{k|k}^1)L'\Lambda^{-1}B^{1'}B^1\Lambda^{-1}L+Q_\omega].\label{new18}
\end{align}
Since $\Sigma_{k|k}^1$ and $\Sigma_{k|k}^2$ are convergent, the second term is convergent. Thus, the system (\ref{13}) is bounded in the mean square sense if and only if the linear system
\begin{align}
\alpha_{k+1}=(A-B\Gamma^{-1}M)\alpha_k,\label{new19}
\end{align}
is stable in the mean square sense, with initial value $\alpha_0=x_0$.

Now we shall show that the system (\ref{new19}) is stable in the mean square sense. To this end, we rewrite (\ref{new10}) as
\begin{align}
\nonumber P&=M'\Upsilon^{-1}R\Upsilon^{-1}M+Q\\
             &\quad+(A-B\Upsilon^{-1}M)'P(A-B\Upsilon^{-1}M).\label{new20}
\end{align}
Define the Lyapunov function as
\begin{align}
\nonumber\tilde{V}_k=\mathbb{E}(\alpha_k'P\alpha_k).
\end{align}
In virtue of (\ref{new20}), we have
\begin{align}
&\nonumber\tilde{V}_{k+1}-\tilde{V}_k\\
\nonumber &=\mathbb{E}\{\alpha_k'[(A-B\Upsilon^{-1}M)'P(A-B\Upsilon^{-1}M)-P]\alpha_k\}\\
 &=-\mathbb{E}[\alpha_k'(M'\Upsilon^{-1}R\Upsilon^{-1}M+Q)\alpha_k], \label{new21}
\end{align}
which means that $\tilde{V}_k$ decreases with respect to $k$. Due to the semi-definite positiveness of $P$, it can be known that $\tilde{V}_k$ is bounded below. Thus, $\tilde{V}_k$ is convergent.

Taking summation from $k=0$ to $k=n$ on both sides of (\ref{new21}) yields
\begin{align}
\nonumber &\tilde{V}_{n+1}-\tilde{V}_0=-\sum_{k=0}^n\mathbb{E}[\alpha_k'(M'\Upsilon^{-1}R\Upsilon^{-1}M+Q)\alpha_k].
\end{align}
Letting $n\to\infty$ on both sides of the above equation, we have
\begin{align}
\nonumber&\lim_{n\to\infty}\mathbb{E}(\alpha_{n+1}'P\alpha_{n+1})\\
\nonumber&=\mathbb{E}(\alpha_0'P\alpha_0)-\lim_{n\to\infty}\sum_{k=0}^n\mathbb{E}[\alpha_k'(M'\Upsilon^{-1}R\Upsilon^{-1}M+Q)\alpha_k].
\end{align}
Since $\tilde{V}_k$ is convergent, we have that $\lim_{n\to\infty}\mathbb{E}[\alpha_n'(M'\Upsilon^{-1}R\Upsilon^{-1}M+Q)\alpha_n]=0$. Thus, $\lim_{n\to\infty}\mathbb{E}[\alpha_n'\alpha_n]=0$, i.e., the system (\ref{new19}) is stable in the mean square sense. Hence, the system (\ref{13}) is bounded in the mean square sense.
\end{proof}

\end{document}